\documentclass[12pt]{amsart}
\usepackage{amsmath, amscd}
\usepackage{amsfonts}

\begin{document}
\numberwithin{equation}{section}

\def\1#1{\overline{#1}}
\def\2#1{\widetilde{#1}}
\def\3#1{\widehat{#1}}
\def\4#1{\mathbb{#1}}
\def\5#1{\frak{#1}}
\def\6#1{{\mathcal{#1}}}

\newcommand{\w}{\omega}
\newcommand{\Lie}[1]{\ensuremath{\mathfrak{#1}}}
\newcommand{\LieL}{\Lie{l}}
\newcommand{\LieH}{\Lie{h}}
\newcommand{\LieG}{\Lie{g}}
\newcommand{\de}{\partial}
\newcommand{\R}{\mathbb R}
\newcommand{\FH}{{\sf Fix}(H_p)}
\newcommand{\al}{\alpha}
\newcommand{\tr}{\widetilde{\rho}}
\newcommand{\tz}{\widetilde{\zeta}}
\newcommand{\tk}{\widetilde{C}}
\newcommand{\tv}{\widetilde{\varphi}}
\newcommand{\hv}{\hat{\varphi}}
\newcommand{\tu}{\tilde{u}}
\newcommand{\tF}{\tilde{F}}
\newcommand{\debar}{\overline{\de}}
\newcommand{\Z}{\mathbb Z}
\newcommand{\C}{\mathbb C}
\newcommand{\Po}{\mathbb P}
\newcommand{\zbar}{\overline{z}}
\newcommand{\G}{\mathcal{G}}
\newcommand{\So}{\mathcal{S}}
\newcommand{\Ko}{\mathcal{K}}
\newcommand{\U}{\mathcal{U}}
\newcommand{\B}{\mathbb B}
\newcommand{\NC}{\mathcal N\mathcal C}
\newcommand{\oB}{\overline{\mathbb B}}
\newcommand{\Cur}{\mathcal D}
\newcommand{\Dis}{\mathcal Dis}
\newcommand{\Levi}{\mathcal L}
\newcommand{\SP}{\mathcal SP}
\newcommand{\Sp}{\mathcal Q}
\newcommand{\A}{\mathcal O^{k+\alpha}(\overline{\mathbb D},\C^n)}
\newcommand{\CA}{\mathcal C^{k+\alpha}(\de{\mathbb D},\C^n)}
\newcommand{\Ma}{\mathcal M}
\newcommand{\Ac}{\mathcal O^{k+\alpha}(\overline{\mathbb D},\C^{n}\times\C^{n-1})}
\newcommand{\Acc}{\mathcal O^{k-1+\alpha}(\overline{\mathbb D},\C)}
\newcommand{\Acr}{\mathcal O^{k+\alpha}(\overline{\mathbb D},\R^{n})}
\newcommand{\Co}{\mathcal C}
\newcommand{\Hol}{{\sf Hol}(\mathbb H, \mathbb C)}
\newcommand{\Aut}{{\sf Aut}(\mathbb D)}
\newcommand{\D}{\mathbb D}
\newcommand{\oD}{\overline{\mathbb D}}
\newcommand{\oX}{\overline{X}}
\newcommand{\loc}{L^1_{\rm{loc}}}
\newcommand{\la}{\langle}
\newcommand{\ra}{\rangle}
\newcommand{\thh}{\tilde{h}}
\newcommand{\N}{\mathbb N}
\newcommand{\kd}{\kappa_D}
\newcommand{\Ha}{\mathbb H}
\newcommand{\ps}{{\sf Psh}}
\newcommand{\Hess}{{\sf Hess}}
\newcommand{\subh}{{\sf subh}}
\newcommand{\harm}{{\sf harm}}
\newcommand{\ph}{{\sf Ph}}
\newcommand{\tl}{\tilde{\lambda}}
\newcommand{\gdot}{\stackrel{\cdot}{g}}
\newcommand{\gddot}{\stackrel{\cdot\cdot}{g}}
\newcommand{\fdot}{\stackrel{\cdot}{f}}
\newcommand{\fddot}{\stackrel{\cdot\cdot}{f}}
\def\v{\varphi}
\def\Re{{\sf Re}\,}
\def\Im{{\sf Im}\,}
\def\rk{{\rm rank\,}}
\def\rg{{\sf rg}\,}
\def\Gen{{\sf Gen}(\D)}
\def\Pl{\mathcal P}
\def\br{{\sf BRFP}}

\newtheorem{theorem}{Theorem}[section]
\newtheorem{lemma}[theorem]{Lemma}
\newtheorem{proposition}[theorem]{Proposition}
\newtheorem{corollary}[theorem]{Corollary}

\theoremstyle{definition}
\newtheorem{definition}[theorem]{Definition}
\newtheorem{example}[theorem]{Example}

\theoremstyle{remark}
\newtheorem{remark}[theorem]{Remark}
\numberwithin{equation}{section}

\title[Range at boundary points]{The range of holomorphic maps  at boundary  points}
\author[F. Bracci]{Filippo Bracci*}
\address{F. Bracci: Dipartimento Di Matematica, Universit\`{a} Di Roma \textquotedblleft Tor
Vergata\textquotedblright, Via Della Ricerca Scientifica 1,
00133, Roma, Italy. } \email{fbracci@mat.uniroma2.it}
\author[J. E. Forn\ae ss]{John Erik Forn\ae ss}
\address{J. E. Forn\ae ss: Department of Mathematical Sciences, Norwegian University of Science and Technology
7491 Trondheim, Norway} \email{john.fornass@math.ntnu.no}

\thanks{$^{*}$Partially supported by the ERC grant ``HEVO - Holomorphic Evolution Equations'' n. 277691.}

\begin{abstract}
We prove a boundary version of the open mapping theorem for  holomorphic maps between strongly pseudoconvex domains. That is, we prove  that the local image of a holomorphic map $f:D\to D'$ close to a boundary regular contact point $p\in \de D$  where the Jacobian is bounded from zero along normal non-tangential directions has to eventually contain every cone (and more generally every admissible region) with vertex at $f(p)$.
\end{abstract}

\maketitle

\section{Introduction}

Let $f:\D \to \D$ be holomorphic. If $\lim_{(0,1)\ni r\to 1}f(r)=1$ and $L:=\liminf_{\zeta\to 1}\frac{1-|f(\zeta)|}{1-|\zeta|}<+\infty$, the point $1$ is called a {\sl boundary regular fixed point} for $f$ and, thanks to the classical Julia-Wolff-Carath\'eodory theorem (see, {\sl e.g.} \cite{A}), it follows that $f$ has non-tangential limit $1$ at $1$ and $f'(\zeta)$ has non-tangential limit $L$ at $1$. In particular, $f$ is isogonal at $1$ and hence it maps angles in $\D$ with vertex at $1$ into angles with vertex $1$ and equal amplitude. Such a result has an interesting quantitative interpretation: every angle with vertex at $1$ is eventually contained in the local image of $f$ at $1$. This can be considered a boundary version of the open mapping theorem, and it is the best one can say about the range of one-dimensional mappings close to boundary points.

In higher dimension,  W. Rudin \cite{Ru} for the unit ball  and M. Abate \cite{A, A0, A1} for strongly (pseudo)convex domains generalized from a qualitative point of view the classical Julia-Wolff-Carath\'eodory theorem. Such a theorem can be seen as a description of the possible one-jets for holomorphic mappings from a strongly pseudoconvex domain into another, close to a boundary point which is non-tangentially mapped to another boundary point. In this optic, in \cite{BZ} a full description of all jets of such mappings, given some smooth extension, is provided. However, the question on how big  the local image of the map close to the boundary point really must be, is not answered from the higher version of the  classical Julia-Wolff-Carath\'eodory theorem. The aim of the present paper is precisely to give an answer to such a question.

In order to  state our result, we need to introduce some notations (see Sections \ref{theball} and \ref{spd} for details). Let $D,D'\subset \C^n$ be two bounded strongly pseudoconvex domains with smooth boundary. Let $p\in \de D$ and let $f: D \to D'$ be holomorphic. Assume
\[
\liminf_{z\to p}\frac{\hbox{dist}(f(z),\de D')}{\hbox{dist}(z,\de D)}<+\infty.
\]
Such a condition is  natural  and it is precisely the analogous of the one assumed in the classical Julia-Wolff-Carath\'eodory theorem. The point $p$ is then called a {\sl regular contact point} and Abate's version of the classical Julia-Wolff-Carath\'eodory theorem implies that there exists  $q\in \de D'$ such that $f$ has non-tangential limit $q$ at $p$. Without other assumption, the local image of $f$ close to $p$ can be very thin (although in Corollary \ref{p1} we prove that, for $D=D'=\B^n$ the unit ball of $\C^n$, the local range of $f$ at $p$ has to be Kobayashi asymptotic to every cone with vertex at $q$). This is not very surprising, since the same happens for points inside $D$ whenever the Jacobian of $f$ is zero.

We say that the point $p$ is a {\sl super-regular contact point} provided $\det df_z$ is bounded from zero when $z$ tends to $p$ non-tangentially along the normal direction to $\de D$ at $p$ (see Definition \ref{super}). In particular, if $f$ is of class $C^1$ at $p$, the point $p$ is  super-regular provided $\det df_p\neq 0$.

Let $\B(x,R)$ denote the Euclidean ball of center $x\in \C^n$ and radius $R>0$. Our main result is the following:

\begin{theorem}\label{main-intro}
Let $D, D'\subset \C^n$ be two bounded strongly pseudoconvex domains with smooth boundary. Let $f: D\to D'$ be holomorphic. Assume $p\in \de D$ is a super-regular contact point for $f$. Then there exists a point $q\in \de D'$ such that for every $\eta>0$ and for every  cone $C\subset D'$  with vertex at $q$ there exists $\delta>0$ such that $C\cap \B(q,\delta)\subset f(\B(p,\eta)\cap D)$.
\end{theorem}

The conclusion of Theorem \ref{main-intro} holds more generally for the so-called {\sl admissible} sets, namely, for those sets in $D'$ which are asymptotic to cones in the Kobayashi distance (see, Section \ref{spd} and Theorem \ref{main}).

The proof of Theorem \ref{main-intro} is based on the corresponding theorem for the unit ball. In Section \ref{theball} we study the local image of holomorphic self-maps of the unit ball close to a regular and super-regular fixed point. In particular, the key result is  Theorem \ref{ball0}, a sort of boundary K\"obe $1/4$-theorem, where we prove that the image close to a super-regular fixed point must contain all the Kobayashi balls of a fixed radius which are centered at points of the image of an angle in the normal directions. With such a result at hands, we get Theorem \ref{main-intro} (and its general version for admissible sets Theorem \ref{main}) by suitably embedding strongly pseudoconvex domains into the unit ball, using a recent result by the second named author with E. F. Wold and K. Diederich, see Section \ref{spd}.

As an application of our result, in Theorem \ref{univalentone}, we prove that if $f:D\to D'$ is univalent and $x\in \de D$ is a super-regular contact point for $f$, then for every regular contact point $y\in\de D\setminus\{x\}$ it holds $f(x)\neq f(y)$. Contrarily to the one-dimensional case (where regular and super-regular points coincide), this is the best one can say. In fact, in  Example \ref{exa} we construct a univalent map of the unit ball having $(\pm 1,0,\ldots, 0)$ as regular contact points (but not super-regular) such that $f(1,0,\ldots,0)=f(-1,0,\ldots, 0)$.

\section{The unit ball}\label{theball}

For a point $a\in \C^n$ and $r>0$ let denote by $\B(a,r):=\{z\in \C^n: \|z-a\|<r\}$. As customary, we let $\B^n:=\B(0,1)$ and denote by $e_1=(1,0,\ldots, 0)$.

Let $\pi:\C^n \to \C^n$ be defined as $\pi(z)=\pi(z_1,\ldots, z_n)=(z_1,0,\ldots, 0)$.

Let $k_{\B^n}$ denote the Kobayashi distance in $\B^n$. Recall that
\[
k_{\B^n}(a,b)=\frac{1}{2}\log \frac{1+\|T_a(b)\|}{1-\|T_a(b)\|},
\]
where $T_a:\B^n\to \B^n$ is any automorphism such that $T_a(a)=0$. We will use such an explicit formula in case $b=\pi(a)$. A direct computation from the explicit form of the automorphisms of $\B^n$ (see, {\sl e.g.}, \cite[p. 358]{A} or below) gives
\begin{equation}\label{special}
\|T_a(\pi(a))\|^2=\frac{\|a-\pi(a)\|^2}{1-\|\pi(a)\|^2}.
\end{equation}

For a subset $A\subset \B^n$ and $z\in \B^n$ we let $k_{\B^n}(z,A)=\inf_{w\in A} k_{\B^n}(z,w)$.

Also, for $z\in \B^n$ and  $R>0$ we let $B_k(z,R)$ denote the {\sl Kobayashi ball} of center $z$ and radius $R$.

\begin{definition}
Let $f:\B^n\to \B^n$ be holomorphic. The point $e_1=(1,0,\ldots,0)$ is said to be a {\sl boundary regular fixed point} of $f$ if
\begin{enumerate}
  \item $\al_f(e_1):=\liminf_{z\to e_1}\frac{1-\|f(z)\|}{1-\|z\|}<+\infty$,
  \item $\lim_{(0,1)\ni r\to 1} f(re_1)=e_1$.
\end{enumerate}
\end{definition}

Let $R\geq 1$. The set $\{z\in \B^n: |1-z_1|\leq
R(1-\|z\|)\}$ is a {\sl Kor\'anyi region of vertex $e_1$ and amplitude $R$} (see
\cite[Section 2.2.3]{A}). In \cite[Section 5.4.1]{Ru}   a slightly different but essentially equivalent definition is given and used. In order not to excessively burden the notation, since we are only working at $e_1$, from now on, when we talk about Kor\'anyi regions, we will always mean Kor\'anyi regions of vertex $e_1$.

Let $f: \B^n \to \C^m$ be a holomorphic map. We say that $f$ has {\sl $K$-limit} $L$ at
$e_1$ -- and we write $K\hbox{-}\lim_{z\to e_1}f(z)=L$ -- if for
each sequence $\{z_k\}\subset \B^n$ converging to $e_1$ such that
$\{z_k\}$  belongs eventually to some Kor\'anyi region, it follows
that $f(z_k)\to L$.

Let $M>1$. We denote by $C(M):=\{z\in \B^n: \|e_1-z\|<M(1-\|z\|)\}$ a {\sl cone of vertex $e_1$ and amplitude $M$}. Also, let $M>1, s\in (0,1)$. We say that $f$ has {\sl
non-tangential limit} $L$ at $e_1$ and we write $\angle\lim_{z\to
e_1}f(z)=L$, if for each sequence $\{z_k\}\subset \B^n$ which is eventually contained in a cone of vertex $e_1$ and amplitude $M>1$, it follows that $f(z_k)\to L$.

In dimension one, Kor\'anyi regions and cones are one and the same and in fact,  studying boundary behavior of holomorphic mappings in the unit disc, it is natural to  consider non-tangential limits. However in higher dimension cones are contained in Kor\'anyi regions, but the first are ``too small'' and the latter ``too big'' and one is forced to consider intermediate sets, which can be tangent to the unit ball in complex tangent directions but are ``asymptotic'' in hyperbolic terms.

\begin{definition}
Let  $A\subset \B^n$ be such that $e_1\in \overline{A}$. We say that $A$ is {\sl admissible} at $e_1$ if for every $\epsilon>0$ there exists $\delta>0$ and $M>1$ such that
\begin{enumerate}
  \item $\pi(A\cap \B(e_1,\delta))\subset C(M)$
  \item $k_{\B^n}(z,\pi(z))<\epsilon$ for every $z\in A\cap \B(e_1,\delta)$.
\end{enumerate}
 \end{definition}

We say that $f$ has {\sl restricted $K$-limit} (or {\sl admissible limit})
$L$ at $e_1$ -- and we write $\angle_K\lim_{z\to e_1}f(z)=L$ -- if
for each sequence $\{z_k\}\subset \B^n$ converging to $e_1$ such that $\{z_k\}$ is admissible at $e_1$ it follows
that $f(z_k)\to L$.

Note that if a sequence $\{z_k\}\subset \B^n$ converging to $e_1$ is admissible at $e_1$, then  $\la z_k, e_1\ra\to 1$ non-tangentially in $\D$ and
\[
\lim_{k\to \infty}k_{\B^n}(z_k,\pi(z_k))=0.
\]
By \eqref{special}, this latter condition is equivalent to
\[
\frac{\|z_k-\la z_k,e_1\ra e_1\|^2}{1-|\la z_k,e_1\ra|^2}\to 0.
\]

One can show that
\[
K\hbox{-}\lim_{z\to e_1}f(z)=L\Longrightarrow \angle_K\lim_{z\to
e_1}f(z)=L\Longrightarrow\angle\lim_{z\to e_1}f(z)=L,
\]
but the converse to any of these implications  is not true in general.

\begin{lemma}\label{cono}
Every cone $C(M)$ in $\B^n$ with vertex $e_1$ and amplitude $M>1$ is admissible at $e_1$.
\end{lemma}

\begin{proof}
Let $C(M)$ be a cone with vertex $e_1$ and amplitude $M>1$.  Let $\epsilon>0$. We want to prove that there exists $\delta>0$ such that
\begin{equation*}
k_{\B^n}(z, \pi(z))<\epsilon\quad \forall z \in C(M)\cap \B(e_1,\delta).
\end{equation*}
By \eqref{special}, this is equivalent to prove  that for each $\eta>0$ there exists $\delta>0$ such that
\begin{equation}\label{dist6}
\frac{\|z-z_1e_1\|^2}{1-|z_1|^2} <\eta \quad \forall z\in C(M)\cap \B(e_1,\delta).
\end{equation}
But,
\[
\frac{\|z-z_1e_1\|^2}{1-|z_1|^2}=\frac{\|z-z_1e_1\|^2}{\|z-e_1\|^2}\frac{\|e_1-z\|}{1-|z_1|}\frac{\|z-e_1\|}{1+|z_1|}\leq M\frac{1-\|z\|}{1-|z_1|}\|e_1-z\|\leq M\|e_1-z\|,
\]
and therefore, if $\delta$ is sufficiently small, \eqref{dist6} follows.
\end{proof}

The following result is due to Rudin \cite{Ru}:

\begin{theorem}[Rudin]\label{RudinJWC}
Let $f:\B^n\to \B^n$ be holomorphic. Suppose that $e_1$ is a boundary regular fixed point for $f$. Then
$K\hbox{-}\lim_{z\to e_1} f(z)=e_1$. Moreover,
\begin{itemize}
\item[(1$^{'}$)] $\langle df_z(e_1), e_1\rangle$ and $\langle df_z(e_h), e_k\rangle$ are bounded in any Kor\'anyi region for $h,k=2,\ldots, n$.
\item[(1$^{''}$)] $\langle df_z(e_j), e_1\rangle/(1-z_1)^{1/2}$ is bounded  in any Kor\'anyi region for $j=2,\ldots, n$.
\item[(1$^{'''}$)] $(1-z_1)^{1/2}\langle df_z(e_1), e_j\rangle$  is bounded  in any Kor\'anyi region for $j=2,\ldots, n$.
\item[(2)] $\angle_K\lim_{z\to e_1}\frac{ 1-\la f(z),e_1\ra}{1-z_1}=\al_f(e_1)$,
\item[(3)] $\angle_K\lim_{z\to e_1}\langle df_z(e_1), e_1\rangle=\al_f(e_1)$,
\item[(4)] $\angle_K\lim_{z\to e_1}\langle df_z(e_j), e_1\rangle=0$ for $j=2,\ldots, n$.
\item[(5)] $\angle_K\lim_{z\to e_1}\frac{\la f(z), e_j\ra}{(1-z_1)^{1/2}}=0$ for $j=2,\ldots, n$.
\item[(6)] $\angle_K\lim_{z\to e_1}(1-z_1)^{1/2}\langle df_z(e_1), e_j\rangle=0$ for $j=2,\ldots, n$.
\end{itemize}
\end{theorem}

\subsection{Automorphisms of $\B^n$}

Automorphisms of the unit ball are linear fractional maps of $\C^n$ which maps $\B^n$ onto $\B^n$ (see, {\sl e.g.} \cite[Section 2.2]{Ru}). In particular they extend holomorphically past the boundary of the unit ball and they map the intersection of the unit ball with a given complex line onto the intersection of the unit ball with another complex line (for details see \cite{BCD}). In what follows we need to consider two special families of automorphisms of the unit ball. Namely, we will use a particular type of {\sl hyperbolic automorphisms} given by
\begin{equation}
\label{hyper}
\Phi_{t_0} (z) = \frac{\left(\cosh \! t_0 \; z_1+\sinh \!{t_0},z_2,\ldots,z_n\right)}{\sinh \! t_0 \;z_1+\cosh\! t_0}
\end{equation}
where $t_0 \in \R \setminus \{ 0 \}$. A direct computation shows that $\Phi_{t_0}(e_1)=e_1$, that $\Phi_{t_0}(-e_1)=-e_1$ (and hence $\Phi_{t_0}(\D\times \{0\})=\D\times \{0\}$) and $\al_{\Phi_{t_0}}(e_1)=e^{-2t_0}$. Moreover, $\det d(\Phi_{t_0})_{0}=(\cosh t_0)^{-(n+1)}$.  Note also that given $r\in (-1,1)$, setting $t_0=\frac{1}{2}\log \frac{1+r}{1-r}$, it follows that $\Phi_{t_0}(0)=re_1$.

Also, we will make use of the  {\sl parabolic } automorphisms of $\B^n$. Recall that an  automorphism $T$ of the unit ball  fixing $e_1$ and with $\al_T(e_1)=1$ is called a {\sl parabolic automorphism} (fixing $e_1$).

Let $\Ha^n:=\{(w_1,w'')\in \C\times \C^{n-1}: \Re w_1 >\|w''\|^2\}$ be the {\sl Siegel domain}. The (generalized) {\sl Cayley transform} $C: \B^n \to \Ha^n$ defined by $C(z_1,z'')=(1+z_1,z'')/(1-z_1)$, where as usual we set $z''=(z_2,\ldots, z_n)$, is a biholomorphism. By \cite[Proposition 4.3]{BCD}, if $T$ is any parabolic automorphism  of $\B^n$ fixing $e_1$ then
\begin{equation}\label{para}
C\circ T \circ C^{-1}(w_1,w'')=(w_1+2\langle Uw'',a\rangle +c, Uw''+a),
\end{equation}
where $U$ is a $(n-1)\times (n-1)$ unitary matrix, $a\in \C^{n-1}$ and $\Re c=\|a\|^2$.

Given $z_0\in \B^n$, let $R(z_0):=\frac{|1-\langle z_0,e_1\rangle |^2}{1-\|z_0\|^2}$. Then, there exists a parabolic automorphism $T_{z_0}$ such that $T_{z_0}(z_0)=\frac{1-R(z_0)}{1+R(z_0)}e_1$. Indeed, set $w_0=C(z_0)$, $U={\sf id}$, $a=-w_0''$, $\Re c=\|w_0''\|^2$ and $\Im c=-\Im (w_0)_1$ in the right-hand side of \eqref{para} and call $\tilde{T}_{z_0}(w)$ such a map. Then
\begin{equation*}
\begin{split}
T_{z_0}(z_0)&=C^{-1}(\tilde{T}_{z_0}(C(z_0))=C^{-1}(\Re (w_0)_1-\|w_0''\|^2, 0,\ldots, 0)\\&=C^{-1}\left(\frac{1-\|z_0\|^2}{|1-\langle z_0, e_1\rangle|^2}, 0,\ldots, 0\right)=\frac{1-R(z_0)}{1+R(z_0)} e_1.
\end{split}
\end{equation*}
A direct computation (or see \cite[eq. (4.2)]{BCD}) shows that for any parabolic automorphism $T$ of the unit ball $\det dT_{e_1}=1$.

\subsection{Range close to boundary regular fixed points}

 Define the {\sl Stolz angle}
\[
K(M,s):=\{\zeta\in \D: \frac{|1-\zeta|}{1-|\zeta|}<M, |1-\zeta|<s\}.
\]
Notice that $K(M,s)$ is the intersection of  a Kor\'anyi region/a cone of vertex $e_1$ and amplitude $M$ with the slice $\C e_1$ and the Euclidean ball $\B(e_1,s)$. In particular, if $g:\B^n\to \C^m$ is a holomorphic map such that $\angle\lim_{z\to e_1}g(z)=L$, then $\lim_{K(M,s)\ni \zeta\to 1}g(\zeta e_1)=L$.

As a matter of notation, let $K(M,s)e_1:=\{z\in \B^n: z=\zeta e_1, \zeta \in K(M,s)\}$

\begin{proposition}\label{p0}
Let $f:\B^n\to \B^n$ be holomorphic. Assume that $e_1$ is a boundary regular fixed point of $f$. Then for every $\epsilon>0$, $M>1$, $t\in (0,1)$   there exists $s=s(\epsilon, M,t)>0$ such that $k_{\B^n}(z,f(K(M,t)e_1)))<\epsilon$ for all $z\in K(M,s)e_1$.
\end{proposition}

\begin{proof}
First of all, we assume that $\al_f(e_1)=1$. For a given $s\in (0,1)$ consider the complex curve $K(M,s)\ni\zeta\mapsto  f(\zeta e_1)\in \B^n$. By Theorem \ref{RudinJWC}, it follows that such a curve is continuous up to the closure of $K(M,s)$.

We claim that, for each $M>1$ and each $\epsilon>0$ there exists $s_0=s_0(M,\epsilon)\in (0,1)$ such that for all fixed $s\in (0,s_0)$ it follows
\begin{equation}\label{dist}
k_{\B^n}(f(\zeta e_1), \pi(f(\zeta e_1)))<\frac{\epsilon}{2}\quad \forall \zeta \in K(M,s).
\end{equation}
Thanks to \eqref{special}, claim \eqref{dist} is equivalent to the claim that
for each $M>1$ and each $\delta>0$ there exists $s_0=s_0(M,\delta)\in (0,1)$ such that for all fixed $s\in (0,s_0)$ it holds
\begin{equation}\label{dist2}
\frac{\|f(\zeta e_1)-\pi(f(\zeta e_1))\|^2}{1-\|\pi(f(\zeta e_1))\|^2} <\delta\quad \forall \zeta \in K(M,s).
\end{equation}
In order to prove this, we write
\[
\begin{split}
\frac{\|f(\zeta e_1)-\pi(f(\zeta e_1))\|^2}{1-\|\pi(f(\zeta e_1))\|^2}&=\frac{\|f(\zeta e_1)-\pi(f(\zeta e_1))\|^2}{|(1-\zeta)^{1/2}|^2}\cdot\frac{|1-\zeta|}{|1-f_1(\zeta e_1)|}\cdot\frac{|1-f_1(\zeta e_1)|}{1-|f_1(\zeta e_1)|^2}\\&=:h_1(\zeta)\cdot h_2(\zeta)\cdot h_3(\zeta).
\end{split}
\]
By Theorem \ref{RudinJWC}.(2) and (5), the functions $\overline{K(M,s)}\ni \zeta\mapsto h_1(\zeta)$ and $K(M,s)\ni \zeta\mapsto h_2(\zeta)$ are (uniformly) continuous on the compact set $\overline{K(M,s)}$. Moreover, $\lim_{K(M,s)\ni \zeta\to 1} h_1(\zeta)=0$, while $\lim_{K(M,s)\ni \zeta\to 1} h_2(\zeta)=1$. As for the function $h_3$, we notice that the function $g(\zeta):=f_1(\zeta e_1)$ is a holomorphic self-map of the unit disc, having $1$ as a boundary regular fixed point because of Theorem \ref{RudinJWC}.(2) and since
\[
\al_g(1)\leq \liminf_{r\to 1} \frac{1-|g(r)|}{1-r} \leq \liminf_{r\to 1} \frac{|1-f_1(re_1)|}{1-r}=1.
\]
Therefore for any $\zeta\in \D$ such that $\frac{|1-\zeta|}{1-|\zeta|}<M$,
\[
\frac{|1-g(\zeta)|}{1-|g(\zeta)|}=\frac{|1-g(\zeta)|}{|1-\zeta|}\frac{|1-\zeta|}{1-|\zeta|}\frac{1-|\zeta|}{1-|g(\zeta)|}
\]
is bounded from above by a constant depending only on $M$  because of the classical Julia-Wolff-Carath\'eodory (that is, Theorem \ref{RudinJWC} for $n=1$). Hence, $h_3(\zeta)$ is bounded from above. Thus, \eqref{dist2} follows.

Next, we claim that for each $M>1$ and each $\epsilon>0$ there exists $s_1=s_1(M,\epsilon)\in (0,1)$ such that for all fixed $s\in (0,s_1)$ it follows
\begin{equation}\label{dist3}
k_{\B^n}(\zeta e_1, \pi(f(\zeta e_1)))<\frac{\epsilon}{2}\quad \forall \zeta \in K(M,s).
\end{equation}
Let $h(\zeta):=f_1(\zeta e_1)$. Thus $h:\D \to \D$ is a holomorphic function and, as before, $1$ is a boundary regular fixed point of $h$. Moreover, from Theorem \ref{RudinJWC}.(2) and since we assumed $\al=1$, it follows that $\al_h(1)=1$. Now,
\[
k_{\B^n}(\zeta e_1, \pi(f(\zeta e_1)))=k_{\D}(\zeta,h(\zeta))=\frac{1}{2}\log\frac{1+\left|\frac{\zeta-h(\zeta)}{1-\overline{\zeta}h(\zeta)} \right|}{1-\left|\frac{\zeta-h(\zeta)}{1-\overline{\zeta}h(\zeta)} \right|},
\]
hence \eqref{dist3} follows as soon as we prove that for each $M>1$ and each $\delta>0$ there exists $s_1=s_1(M,\delta)\in (0,1)$ such that for all fixed $s\in (0,s_1)$ it holds
\begin{equation}\label{dist4}
\left|\frac{\zeta-h(\zeta)}{1-\overline{\zeta}h(\zeta)} \right|<\delta\quad \forall \zeta \in K(M,s).
\end{equation}
Let $d(\zeta):=(1-h(\zeta))/(1-\zeta)$ and $a(\zeta)=(1-\overline{\zeta})/(1-\zeta)$. Thus
\begin{equation}\label{divis}
\left|\frac{\zeta-h(\zeta)}{1-\overline{\zeta}h(\zeta)}\right|=\left|\frac{1-d(\zeta)}{a(\zeta)+\overline{\zeta}d(\zeta)}\right|.
\end{equation}
By Theorem \ref{RudinJWC} for $n=1$ it follows that $d(\zeta)$ is (uniformly) continuous on the compact set $\overline{K(M,s)}$ and $d(\zeta)\to 1$ as $\overline{K(M,s)}\ni \zeta \to 1$. Moreover, if $s<1/2$ then $\Re \zeta>1/2$ and then $\Re a(\zeta)>-1/4$. Hence, if $s$ is sufficiently small, it follows that the real part of the denominator of the right hand side of \eqref{divis} is bounded from below by a constant $c>0$ for all $\zeta\in K(M,s)$. As a result, if we choose $s$ sufficiently small, \eqref{dist4} holds.

Now, fix $\epsilon>0$, $M>1$ and $t\in (0,1)$.  Choose $s_0, s_1$ such that \eqref{dist}, \eqref{dist3} hold, and let $s:=\min\{t,s_0,s_1\}$. Let $z\in K(M,s)e_1$, hence
\[
k_{\B^n}(f(K(M,s)e_1),z)\leq k_{\B^n}(f(z), z)\leq k_{\B^n}(z, \pi(f(z)))
+k_{\B^n}(\pi(f(z)), f(z))<\epsilon.
\]
To end up the proof, we need to consider the case $\al:=\al_f(e_1)\neq 1$. Let $\Phi$ be an {\sl hyperbolic} automorphism of $\B^n$ of type \eqref{hyper} with $\al_{\Phi}(e_1)=\al^{-1}$. In particular, since the first component is a M\"obius transformation of $\C$, it follows that for all $s\in (0,1)$ there exists $s'(s)\in (0,1)$ such that $\Phi(K(M,s'(s))e_1)=K(M,s)e_1$. Using Theorem \ref{RudinJWC} it is easy to see that the holomorphic self-map $g:=f\circ \Phi$ of $\B^n$ has the property that $e_1$ is a boundary regular fixed point of $g$ and $\al_g(e_1)=1$. Let $\epsilon>0, t\in (0,1)$ and $M>1$. Let $t'(t)\in (0,1)$ be such that $\Phi(K(M,t'(t))e_1)=K(M,t)e_1$. We apply the result already proven to $g$, finding
$s>0$ (in fact, for what we have proven, $s=t'(t)$) such that $k_{\B^n}(z,g(K(M,t'(t))e_1))<\epsilon$ for all $z\in K(M,s)$. By construction
$g(K(M,t'(t))e_1)=f(\Phi(K(M,t'(t))e_1)= f(K(M,t)e_1)$ and therefore $k_{\B^n}(z,f(K(M,t)e_1)<\epsilon$ for all $z\in K(M,s)e_1$, as wanted.
\end{proof}

\begin{corollary}\label{p1}
Let $f:\B^n\to \B^n$ be holomorphic. Assume that $e_1$ is a boundary regular fixed point of $f$. Then for every $\epsilon>0$ and $t\in (0,1)$ and for every  $A\subset \B^n$ admissible at $e_1$ there exists $\delta>0$ such that $k_{\B^n}(z,f(K(M,t)e_1))<\epsilon$ for all $z\in A\cap \B(e_1,\delta)$.
\end{corollary}

\begin{proof}
Fix $\epsilon>0$ and $t\in (0,1)$. Let $\delta_1>0$ be such that $\pi(A\cap \B(e_1,\delta_1))\subset K(M,s_1)e_1$ for some $s_1\in (0,1)$ and $k_{\B^n}(z,\pi(z))\leq \epsilon/2$ for all $z\in A\cap \B(e_1,\delta_1)$. Let $s_0\in (0,1)$ be such that $k_{\B^n}(z, f(K(M,t)))\leq \epsilon/2$ for all $z\in K(M,s_0)e_1$. Let  $\delta:=\min\{t,s_0,\delta_1\}$. Let $z\in A\cap \B(e_1,\delta)\subset A\cap \B(e_1,\delta_1)$. Hence $\pi(z)\in  K(M,s_0)$ and
\[
k_{\B^n}(z,f(K(M,t)e_1))\leq k_{\B^n}(z, \pi(z))+k_{\B^n}(\pi(z),f(K(M,t)e_1)) <\epsilon,
\]
and the proof is completed.
\end{proof}

\begin{example}
In general, Corollary \ref{p1} is the best one can say, namely, the image of a holomorphic map close to a boundary regular fixed point can be very thin: just consider the simple example $f:\B^2\to \B^2$ given by $f(z_1,z_2)=(z_1,0)$. Even assuming univalency the situation does not improve:  the map $f(z_1,z_2)=(z_1, \frac{1}{4}z_2(1-z_1)^2)$ is a univalent self-map of $\B^2$ having a boundary regular fixed point at $e_1$ but its image does not contain eventually any cone with vertex at $e_1$.
\end{example}

The problem with the previous examples is that the Jacobian of the maps at $e_1$ is zero. In the next section we show that if this is not the case, then the range of the map close to the boundary point is ``fat''.

\subsection{Range close to boundary super-regular fixed points}

\begin{definition}
Let $f:\B^n\to \B^n$ be holomorphic. The point $e_1$ is  a {\sl boundary super-regular fixed point} of $f$ if it is a boundary regular fixed point of $f$ and moreover for each $M>1$ there exists $c=c(M)>0$ such that for any sequence $\{\zeta_k\}\subset K(M,2)$ which tends  to $1$ it holds $\liminf_{k\to \infty}|\det (df_{\zeta_k e_1})|\geq c$.
\end{definition}

\begin{remark}
Let $f:\B^n\to \B^n$ be holomorphic. Assume $e_1$ is a boundary regular fixed point of $f$. As shown in \cite[Example p.183]{Ru}, the radial limit of $\det (df_z)$ at $e_1$ may not exist. In fact, by Theorem \ref{RudinJWC} it follows that $\det df_z$ is bounded in every Kor\'anyi region and $\check{\hbox{C}}$irca's theorem
\cite[Theorem 8.4.8]{Ru} implies  that if $\lim_{r\to 1} \det (df_{re_1})$ exists, then $\angle_K\lim_{z\to e_1} \det df_z$ exists. If this is the case, $e_1$ is a boundary super-regular fixed point if and only if $\angle_K\lim_{z\to e_1} \det df_z\neq 0$.
\end{remark}

\begin{lemma}\label{Kobe}
Let $c>0$. Let $\mathcal F_c:=\{f: \B^n \to \B^n \hbox{\ holomorphic}: f(0)=0, |\det df_0|\geq c\}$. Let $t'\in (0,1)$. Then there exists $r'=r'(c,t')>0$ such that $\B(0,r')\subset f(\B(0,t'))$ for all $f\in \mathcal F_c$.
\end{lemma}

\begin{proof}
Assume this is not the case. Then for all $n\in \N$ such that $1/n<t'$ there exists $z_n\in \B(0,t')$ with $|z_n|<1/n$ and $f_n\in \mathcal F_c$ such that $z_n\not\in f(\B(0,t'))$. Since $\|f_n(z)\|<1$ for all $z\in \B^n$, and $f_n(0)=0$, $\{f_n\}$ is a normal family and we can assume it is uniformly convergent on compacta to a holomorphic map $f:\B^n \to \B^n$ such that $f(0)=0$. Since $d(f_n)_0\to df_0$, the map $f\in \mathcal F_c$.  Therefore there exists $\delta>0$ such that $\B(0,\delta)\subset f(\B(0,t'))$. Hence, eventually $\B(0,\delta/2)\subset f_n(\B(0,t'))$, which contradicts the choice of $\{z_n\}$.
\end{proof}

\begin{theorem}\label{ball0}
Let $f:\B^n\to \B^n$ be holomorphic. Assume that $e_1$ is a boundary super-regular fixed point of $f$. Then for every $\eta>0$ and $M>1$ there exist $s\in (0,1)$ and $r>0$ such that
\begin{equation}\label{claimB}
\bigcup_{\zeta \in K(M,s)} B_k(f(\zeta e_1), r))\subset f(\B^n\cap\B(e_1,\eta)),
\end{equation}
where $B_k(x,R)$ denotes the Kobayashi ball of center $x\in \B^n$ and radius $R>0$.
\end{theorem}

\begin{proof}
For $z\in \B^n$, let us set
\[
R(z)=\frac{|1-z_1|^2}{1-\|z\|^2}, \quad r(z)=\frac{1-R(z)}{1+R(z)}.
\]
Let $\zeta\in K(M,s)$, for some $s\in (0,1)$ to be chosen later. Let $S$ be a parabolic automorphism of $\B^n$ fixing $e_1$ such that $S(\zeta e_1)=r(\zeta e_1)e_1$. Next, let $\Phi_{t_0}$ be a hyperbolic automorphism of the form \eqref{hyper} with $t_0=\frac{1}{2}\log \frac{1+r(\zeta e_1)}{1-r(\zeta e_1)}=-\log R(\zeta e_1)^{1/2}$. Hence $\Phi_{t_0}(0)=r(\zeta e_1) e_1$.

let $T$ be a parabolic automorphism of $\B^n$ fixing $e_1$ such that $T(f(\zeta e_1))=r(f(\zeta e_1)) e_1$, and let $\Phi_{t_1}$ be a hyperbolic automorphism of the form \eqref{hyper} with $t_1=\frac{1}{2}\log \frac{1+r(f(\zeta e_1))}{1-r(f(\zeta e_1))}=-\log R(f(\zeta e_1))^{1/2}$. Hence $\Phi_{t_1}(0)=r(f(\zeta e_1))e_1$.

Now, let $g^\zeta:=  \Phi^{-1}_{t_1} \circ T\circ f\circ  S^{-1}\circ \Phi_{t_0}$. By construction, $g^\zeta:\B^n \to \B^n$ is holomorphic and $g^\zeta(0)=0$. Moreover,
\begin{equation}\label{differential}
\det dg^\zeta_0=\det (\Phi^{-1}_{t_1})_{r(f(\zeta e_1))e_1}\cdot \det dT_{f(\zeta e_1)} \cdot \det df_{\zeta e_1} \cdot \det dS^{-1}_{r(\zeta e_1)e_1}\cdot \det d(\Phi_{t_0})_0.
\end{equation}
Since $f$ is (uniformly) continuous on $\overline{K(M,s)}$ for any $s\in (0,1)$ and $\lim_{K(M,s)\ni \zeta \to 1}f(\zeta e_1)=e_1$, and since $\det dT_{e_1}=\det dS_{e_1}=1$, it follows that there exists $s_0\in (0,1)$ such that for all $s\in (0,s_0)$ and for each $\zeta \in K(M,s)$ it holds
\begin{equation}
\label{estimo1}
|\det dS_{\zeta e_1}|\leq 2, \quad |\det dT_{f(\zeta e_1)}|\geq \frac{1}{2}.
\end{equation}
Therefore, since $\det dS^{-1}_{r(\zeta e_1)e_1}=\det dS^{-1}_{S(\zeta e_1)}=(\det dS_{\zeta e_1})^{-1}$, it follows that
\begin{equation}
\label{estimo2}
|\det  dS^{-1}_{r(\zeta e_1)e_1}|\geq \frac{1}{2}.
\end{equation}
Also,
\begin{equation}
\label{estimo3}
\begin{split}
&\det (\Phi^{-1}_{t_1})_{r(f(\zeta e_1))e_1}=\det d(\Phi^{-1}_{t_1})_{\Phi_{t_1}(0)}=(\det (d\Phi_{t_1})_0)^{-1}=(\cosh t_1)^{n+1},\\ &\det d(\Phi_{t_0})_0=(\cosh t_0)^{-(n+1)}.
\end{split}
\end{equation}
Finally, taking into account that  $e_1$ is a boundary super-regular fixed point of $f$, there exists $s_1\in (0,s_0)$, and  $c>0$ such that $|df_{\zeta e_1}|\geq c$ for all $\zeta \in K(M,s)\subset K(M,2)$ for all $s\in (0,s_1)$. From \eqref{differential}, \eqref{estimo1}, \eqref{estimo2}, \eqref{estimo3} we obtain
\[
|\det dg^\zeta_0|\geq \frac{c}{4}\left(\frac{\cosh t_1}{\cosh t_0}\right)^{n+1}.
\]
A direct computation shows that, for $\zeta \in K(M,s)$,
\begin{equation}\label{bound}
\begin{split}
\frac{\cosh t_1}{\cosh t_0}&=\frac{1+R(f(\zeta e_1)}{1+R(\zeta e_1)}\cdot\left(\frac{R(\zeta e_1)}{R(f(\zeta e_1))}\right)^{1/2}
\\&=\frac{\frac{1-\|f(\zeta e_1\|^2}{1-|\zeta|^2}+\frac{|1-f_1(\zeta e_1)|^2}{1-|\zeta|^2}}{1+\frac{|1-\zeta|^2}{1-|\zeta|^2}}\frac{1-|\zeta|^2}{1-\|f(\zeta e_1)\|^2}\left(\frac{|1-\zeta|^2}{|1-f_1(\zeta e_1)|^2}\frac{1-\|f(\zeta e_1)\|^2}{1-|\zeta|^2}  \right)^{1/2}\\
&\geq
\frac{\frac{1-\|f(\zeta e_1\|^2}{1-|\zeta|^2}}{1+M|1-\zeta|}\frac{1-|\zeta|^2}{1-\|f(\zeta e_1)\|^2}\frac{|1-\zeta|}{|1-f_1(\zeta e_1)|}\left(\frac{1-\|f(\zeta e_1)\|^2}{1-|\zeta|^2}  \right)^{1/2}\\&
\geq \frac{1}{1+Ms}\frac{|1-\zeta|}{|1-f_1(\zeta e_1)|}\left(\frac{1-\|f(\zeta e_1)\|^2}{1-|\zeta|^2}  \right)^{1/2}
\end{split}
\end{equation}
Let $\al:=\al_f(e_1)$. Since $\al=\liminf_{z\to e_1}(1-\|f(z)\|)/(1-\|z\|)$, there exists $s_2\in (0,s_1)$ such that for all $s\in(0,s_1)$, it holds $\frac{1-\|f(\zeta e_1)\|^2}{1-|\zeta|^2}\geq \al/4$ for all $\zeta \in K(M,s)$.

Also, since $e_1$ is a boundary regular fixed point for $f$ by Theorem \ref{RudinJWC}.(2) it follows that there exists $s_3\in (0,s_2)$ such that for all $s\in (0,s_3)$, it holds $\frac{ |1-f_1(\zeta e_1)|}{|1-\zeta|}\leq 2\al$ for all $\zeta \in K(M,s)$.
Using the previous estimates in \eqref{bound}, we conclude that there exists $c'>0$ such that $\left(\frac{\cosh t_1}{\cosh t_0}\right)^{n+1}\geq c'$ and hence $|\det dg^\zeta_0|\geq c'':=cc'/4$.

Therefore, for all  $\zeta\in K(M,s_3)$, the map $g^\zeta\in \mathcal F_{c''}$, where $\mathcal F_{c''}$ is defined in Lemma \ref{Kobe}. Thus, by Lemma \ref{Kobe}, for every $t'\in (0,1)$ there exists $r'\in (0,1)$ such that $\B(0,r')\subset g^\zeta (\B(0,t'))$ for all $\zeta \in K(M,s_3)$.

Now,   for $a'>0$ set $a=\frac{1}{2}\log \frac{1+a'}{1-a'}$. Since automorphisms of $\B^n$ are isometries for $k_{\B^n}$ and $\B(0,a')=B_k(0,a)$, by the very definition of $g^\zeta$ it follows that for all $\zeta \in K(M,s_3)$
\[
B_k(f(\zeta e_1), r)=T^{-1}(\Phi_{t_1}(\B(0,r'))\subset f(S^{-1}(\Phi_{t_0}(\B(0,t')))=f(B_k(\zeta e_1, t)).
\]
Finally, let $\eta>0$ be as in claim \eqref{claimB} and choose $s\in (0,s_3)$ such that for each $\zeta \in K(M,s)$, it holds $B_k(\zeta e_1,t)\subset \B^n\cap \B(e_1,\eta)$. This can be done because the Euclidean diameter of a Kobayashi ball of fixed radius tends to zero when the center tends to $e_1$. Hence for any $\zeta \in K(M,s)$,
\[
B_k(f(\zeta e_1), r)\subset f(B_k(\zeta e_1, t))\subset f(\B^n\cap \B(e_1,\eta)),
\] and \eqref{claimB} is proved.
\end{proof}

As a consequence, we have that any subset $A\subset \B^n$ which is non-tangentially asymptotic at $e_1$ is eventually contained in the range close to a boundary super-regular fixed point:

\begin{theorem}\label{ball}
Let $f:\B^n\to \B^n$ be holomorphic. Assume that $e_1$ is a boundary super-regular fixed point of $f$. Then for every $\eta>0$ and for every $A\subset \B^n$ admissible at $e_1$  there exists $\delta>0$ such that $A\cap \B(e_1,\delta)\subset f(\B(e_1,\eta)\cap \B^n)$.
\end{theorem}

\begin{proof}
Let $\eta>0$. Let $s,r>0$ be given by Theorem \ref{ball0}. Let $0<\epsilon<r$. By Corollary \ref{p1} there exists $\delta>0$ such that $k_{\B^n}(z,f(K(M,s)e_1))<\epsilon$ for all $z\in A\cap \B(e_1,\delta)$. Therefore, any $z\in A\cap \B(e_1,\delta)$ is contained in a Kobayashi ball centered at some point of $f(K(M,s)e_1)$ and with radius at most $\epsilon<r$, hence by Theorem \ref{ball0}, it is contained in $f(\B^n\cap\B(e_1,\eta))$, as stated.
\end{proof}

\section{Strongly pseudoconvex domains}\label{spd}

Let $D\subset \C^n$ be a bounded strongly pseudoconvex domain with smooth boundary. Let $p\in \de D$. We denote by $\nu^D_p$ the outer unit normal vector to $\de D$ at $p$. For $z,w\in \C^n$, let $d(z,w)=\|z-w\|$. A {\sl cone} $C(p,M)$ with vertex $p$ and amplitude $M>1$ is defined as
\[
C(p,M):=\{z\in D: d(p,z)<M d(z,\de D)\}.
\]
We say that a sequence $\{z_k\}\subset D$ tends to $p$ {\sl normally non-tangentially} if $z_k\to p$ and for all $k\in \N$ there exists $\zeta_k\in \C$ such that $z_k=\zeta_k \nu^D_p$ and there exists $M>1$  such that $\{z_k\}\subset C(p,M)$.

\begin{definition}\label{super}
Let $D, D'\subset \C^n$ be two bounded strongly pseudoconvex domains with smooth boundary. Let $f: D\to D'$ be holomorphic. A point $p\in \de D$ is said to be a {\sl regular contact point} for $f$ if
\[
\liminf_{D\ni z\to p}\frac{d(f(z), \de D')}{d(z,\de D)}<+\infty.
\]
The point $p$ is said to be a {\sl super-regular contact point} for $f$ if it is a regular contact point for $f$ and moreover for every $M>1$ there exists $c=c(M)>0$ such that
\[
\liminf_{k\to \infty} |\det df_{z_k}|\geq c,
\]
for every sequence $\{z_k\}\subset C(M,p)$ which converges normally non-tangentially to $p$.
\end{definition}

In  \cite[Theorem 0.2]{A1} M. Abate proved that if $f:D\to D'$ is a holomorphic map between two bounded strongly pseudoconvex domains and $p\in \de D$ is a regular contact point, then an analogous of Rudin's theorem \ref{RudinJWC} holds. In the proof of our main Theorem \ref{main} we will make use of part of Abate's theorem. For the reader convenience we state here what we need from \cite[Theorem 0.2]{A1}:

\begin{theorem}[Abate]\label{AbateJWC}
Let $D, D'\subset \C^n$ be two bounded strongly pseudoconvex domains with smooth boundary. Let $f: D\to D'$ be holomorphic. Assume $p\in \de D$ is a regular contact point for $f$. Then there exists $q\in \de D'$ such that
\begin{enumerate}
  \item $\lim_{\in (0,1)\ni r\to 0}f(p-r\nu^D_p)=q$,
  \item $\lim_{\in (0,1)\ni r\to 0}\langle df_{p-r\nu^D_p}(\nu_p^D),\nu_q^{D'}\rangle$ exists finitely.
\end{enumerate}
\end{theorem}

Also, we will make use of the following recent result by the second named author with E. F. Wold and K. Diederich \cite{DFW}:

\begin{theorem}\label{embedding}
Let $D\subset \C^n$ be a bounded strongly pseudoconvex domain with smooth boundary and let $q\in \de D$. Then there exists a biholomorphism $\Phi$ from an open neighborhood of $\overline{D}$ such that $\Phi(D)\subset \B^n$ and $\Phi(q)=e_1$.
\end{theorem}

\begin{definition}
Let $D\subset \C^n$ be a bounded strongly pseudoconvex domain with smooth boundary. A subset $A\subset D$ such that $p\in \overline{A}$ is {\sl admissible} at $p\in \de D$ if for every $\epsilon>0$ there exist $\delta>0$ and $M>1$ such that
\begin{enumerate}
  \item $\langle z, \nu_p^D\rangle \nu_p^D\in  C(p,M)$ for all $z\in A\cap \B(p,\delta)$,
  \item $k_D(z,\pi(z))<\epsilon$ for all $z\in  A\cap \B(p,\delta)$,
\end{enumerate}
where $k_D$ denotes the Kobayashi distance in $D$.
\end{definition}

\begin{lemma}
Let $D\subset \C^n$ be a bounded strongly pseudoconvex domain with smooth boundary. Any cone $C(p,M)\subset D$ with vertex $p\in \de D$ and amplitude $M>1$ is admissible at $p$.
\end{lemma}
\begin{proof}
It is enough to note that if $B\subset D$ is a Euclidean ball tangent to $p$, then every cone $C(p,M)$ is eventually contained in $B$ and $C(p,M)\cap B$ is contained in a cone in $B$ with vertex $p$. Hence by Lemma \ref{cono} it is admissible in $B$. By the monotonicity property of the Kobayashi distance, $k_D\leq k_{B}$ and hence the result follows.
\end{proof}

Our main result is the following:

\begin{theorem}\label{main}
Let $D, D'\subset \C^n$ be two bounded strongly pseudoconvex domains with smooth boundary. Let $f: D\to D'$ be holomorphic. Assume $p\in \de D$ is a super-regular contact point for $f$. Then there exists a point $q\in \de D'$ such that for every $\eta>0$ and for every  $A\subset D'$ admissible at $q$ there exists $\delta>0$ such that $A\cap \B(q,\delta)\subset f(\B(p,\eta)\cap D)$.
\end{theorem}

\begin{proof} Up to rotations and dilations, we can assume that $p=e_1$ and $\B^n\subset D$. Moreover, by Theorem \ref{embedding}, we can assume that $q=e_1$ and $D'\subset \B^n$. In particular, with these choices, $\nu_p^D=\nu_q^{D'}=e_1$. Hence, the map $g:=f|_{\B^n}:\B^n\subset D\to D'\subset \B^n$ is a holomorphic self-map of $\B^n$ and by Theorem \ref{AbateJWC}.(1), $\lim_{\in (0,1)\ni r\to 1}g(re_1)=e_1$.
We claim that $e_1$ is a boundary super-regular fixed point for $g$.  Since $e_1$ is a super-regular contact point for $f: D \to D'$, in order to show that $e_1$ is a boundary super-regular fixed point for $g$, we only need to show that
\begin{equation}\label{limo}
\liminf_{z\to e_1}\frac{1-\|g(z)\|}{1-\|z\|}<+\infty.
\end{equation}
By Theorem \ref{AbateJWC}.(2),  we have
\[
\lim_{(0,1)\ni r\to 1}\langle dg_{re_1}(e_1),e_1\rangle=\lim_{ (0,1)\ni r\to 0}\langle df_{p-r\nu^D_p}(\nu_p^D),\nu_q^{D'}\rangle=L,
\]
for some $L\in \C$. Let $h(\zeta):=\langle g(\zeta e_1), e_1\rangle$. Hence $h$ is a holomorphic self-map of $\D$ and $\lim_{r\to 1}h(r)=1$, $\lim_{r\to 1}h'(r)=\lim_{(0,1)\ni r\to 1}\langle dg_{re_1}(e_1),e_1\rangle=L$. By the mean value theorem applied to the real and imaginary part of $(0,1)\ni r\mapsto h(r)$, it follows that
\[
\lim_{r\to 1} \frac{1-h(r)}{1-r}=L.
\]
Therefore
\[
\begin{split}
\liminf_{z\to e_1}\frac{1-\|g(z)\|}{1-\|z\|}&\leq \liminf_{(0,1)\ni r\to 1}\frac{1-\|g(re_1)\|}{1-r}\leq \liminf_{(0,1)\ni r\to 1}\frac{1-|\langle g(re_1), e_1\rangle |}{1-r}\\& \leq \liminf_{(0,1)\ni r\to 1}\frac{|1-h(r) |}{1-r}=|L|<+\infty,
\end{split}
\]
as needed.

Now, note that $A$ is admissible at $e_1$ not only as a subset of $D$ but also as a subset of $\B^n$. Indeed, its projection into $\C e_1$ is eventually contained in a cone in $D$ and hence in $\B^n$, also, the Kobayashi distance is monotonic and then  $k_{\B^n}(z,\pi(z))\leq k_D(z,\pi(z))<\epsilon$ for all $z\in A$ close to $e_1$. Thus, we can apply Theorem \ref{ball} to $g$ and, since $g(\B^n\cap \B(e_1,\eta))\subset f(D\cap \B(e_1,\eta))$ for all $\eta>0$ we get also the result for $f$.
\end{proof}

\begin{remark}
On the one hand, it would be interesting to give a direct proof of Theorem \ref{main} without using the embedding Theorem \ref{embedding}, but using instead the full Abate's version of Rudin's theorem  for strongly pseudoconvex domains  and suitably adapting our  proof of Theorem \ref{ball}. However, aside using Rudin's theorem, our argument in the proof of Theorem \ref{ball} is strongly based on the existence of a family of good automorphisms of $\B^n$, and it is not clear how to bypass such an argument for strongly pseudoconvex domains.

On the other hand,  it would be interesting to prove \eqref{limo} without using Abate's Theorem \ref{AbateJWC}. If this were possible, the previous method would allow to prove Abate's version of Rudin's theorem  for strongly pseudoconvex domains directly by means of Rudin's theorem.
\end{remark}

\section{Images of regular contact points by univalent mappings}

\begin{theorem}\label{univalentone}
Let $D,D'\subset \C^n$ be  bounded strongly pseudoconvex domains with smooth boundary. Let $f: D \to D'$ be a univalent map and assume $x\in \de D$ is a super-regular contact point for $f$, and denote by $f(x)\in \de D'$ the non-tangential limit of $f$ at $x$. If $y\in \de D$ is a regular fixed point for $f$ then $f(x)\neq f(y)$.
\end{theorem}

\begin{proof}
Assume by contradiction that $f(x)=f(y)$. Up to rotations and dilations, we can assume that $y=e_1$ and $\B^n\subset D$. Moreover, by Theorem \ref{embedding}, we can assume that $f(x)=e_1$ and $D'\subset \B^n$. Let $g:=f|_{\B^n}:\B^n\subset D \to D'\subset \B^n$, seen as a holomorphic self-map of the unit ball. Arguing as in the proof of Theorem \ref{main}, we see that $e_1$ is a boundary regular fixed point for $g$. Hence, the curve $\Gamma:(0,1)\ni r\mapsto g(re_1)$ is admissible at $e_1$ by  \cite[Lemma 2.2]{Br}. In fact,  $g$ maps Kor\'anyi regions into Kor\'anyi regions  because $e_1$ is a boundary regular fixed point (see, {\sl e.g.}, \cite[Theorem 8.5.4]{Ru}), hence $(0,1)\ni r\mapsto \pi(g(r e_1))$ is non-tangential and it is asymptotic by  \eqref{dist}.

Let $B\subset D$ be an Euclidean ball tangent to $\de D$ at $p$. Consider $h=f|_{B}:B\subset D\to D'\subset \B^n$ as a holomorphic map from $B$ to $\B^n$. By Theorem \ref{main} it follows that for every $\eta>0$ there exists $\delta>0$ such that $\Gamma\cap \B(e_1,\delta)\subset h(B\cap \B(p,\eta))\subset f(D\cap \B(p,\eta))$. For $\eta$ small, this contradicts the univalency of $f$.
\end{proof}

In dimension one, regular contact points and super-regular contact points are one and the same, and Theorem \ref{univalentone} recovers \cite[Lemma 8.2]{CP}. In higher dimensions however, a univalent mapping can send two different regular contact points (but not super-regular) to the same point, as the following example shows:

\begin{example}\label{exa}
 Let $\v:\D \to \D$ be a holomorphic mapping with the following properties:
\begin{enumerate}
  \item $\v(\D)\subset \{\zeta\in \D: |\zeta-1|<1/2\}$,
  \item $\v$ extends smoothly up to the boundary and $\v(-1)=\v(1)=1$,
  \item there exist $0<r_{-1},r_1<1/2$ such that the two open discs $A=D(-1,r_{-1}), B=D(1,r_1)$ are disjoint and $\v$ restricted to $\D\setminus A$ and restricted to $\D\setminus B$ is univalent.
\end{enumerate}
Such a map can be constructed by taking the Riemann mapping from the unit disc to an open simply connected Riemann surface with smooth boundary in $\C^4$ which sits on a helicoid a project the image back to $\C$.

Now, let $C:\D \to \Ha:=\{\zeta\in \C: \Re \zeta>0\}$ be the Cayley transform $C(\zeta)=(1+\zeta)/(1-\zeta)$. Let $\phi:=C\circ \v \circ C^{-1}$. Set $R_0:=r_{-1}/(2-r_{-1})$, $R_\infty=(2-r_1)/r_1$ and let $D(x,R)$ denote the Euclidean disc of center $x\in \C$ and radius $R>0$. Then $U=C(A)=\Ha\cap D(0,R_0)$ and $V=C(B)=\Ha\setminus\overline{D(0,R_\infty)}$. The set $\phi(\Ha)$ is contained in $\{\zeta\in \C:\Re \zeta> 3\}$ and $\phi$ is univalent in $\Ha\setminus U$ and in $\Ha\setminus V$. Moreover, $\lim_{\Ha \ni \zeta\to 0}|\phi(\zeta)|=\lim_{\Ha \ni \zeta\to \infty}|\phi(\zeta)|=\infty$. Let
\[
\al:=\lim_{(0,\infty)\ni r\to \infty}\frac{\phi(r)}{r}.
\]
Such a limit exists, $\al>0$ and moreover $\Re \phi(\zeta)>\al \Re \zeta$ for all $\zeta\in \Ha$  by the Julia-Wolff-Carath\'eodory theorem (that is Theorem \ref{RudinJWC} for $n=1$ in its right-half plane formulation, see, {\sl e.g.} \cite[Corollary 1.2.12]{A}).

Let $\Ha^2:=\{(z,w)\in \C^2: \Re z>\|w\|^2\}$. For $\delta,\epsilon>0$ we define a map $\Phi:\Ha^2\to \C^2$ as follows
\[
\Phi(z,w)=\left(\phi(z), \frac{\sqrt{\al}\delta w}{1+\phi(z)}+\frac{\epsilon z}{1+z} \right).
\]
The map $\Phi$ is clearly holomorphic, and we claim that, for sufficiently small $\delta, \epsilon>0$ it maps $\Ha^2$ into $\Ha^2$ and it is univalent. Assume this is the case, let $C(z,w):=(1+z,w)/(1-z)$ be the Cayley transform from $\B^2$ onto $\Ha^2$. Let $f:= C^{-1} \circ\Phi\circ C$. Then $f:\B^2\to \B^2$ is a univalent map. Moreover, $\pm e_1$ are regular contact points for $f$. Indeed, $f_1(z,w)=\v(z)$ which is smooth at $\pm 1$ and $\v(\pm 1)=1$. Hence
\[
\liminf_{r\to 1}\frac{1-\|f(\pm re_1)\|}{1-r}\leq \liminf_{r\to 1}\frac{|1-\v(\pm r)|}{1-r}=|\v'(\pm 1)|<+\infty.
\]
Also,  by construction $f(\pm e_1)=e_1$.

We are left to show that we can find $\delta,\epsilon>0$ such that
\begin{enumerate}
  \item $\Phi(\Ha^2)\subset \Ha^2$,
  \item $\Phi$ is injective.
\end{enumerate}
(1) Let $(z,w)\in \Ha^2$. Since $|1+\phi(z)|\geq 1+\Re \phi(z)\geq 1+3=4$ and $|z|/(|1+z|)\leq 1$,
\[
\frac{\sqrt{\al}\delta |w|}{|1+\phi(z)|}\leq \frac{\sqrt{\al}\delta |w|}{4}, \quad \frac{\epsilon |z|}{|1+z|}\leq \epsilon.
\]
Let $\epsilon>0$ be such that $\epsilon^2< 3/4$, and let $\delta>0$ be such that $\delta<2$. Since $3<\Re \phi(z)$ and $\Re \phi(z)\geq \al \Re z$,
\begin{equation*}
\begin{split}
\left|\frac{\sqrt{\al}\delta w}{1+\phi(z)}+\frac{\epsilon z}{1+z}  \right|^2&\leq 2\left(\frac{\al\delta^2|w|^2}{|1+\phi(z)|^2}+
\frac{\epsilon^2 |z|^2}{|1+z|^2} \right)\leq 2\left(\frac{\al \delta^2|w|^2}{(1+3)^2}+\epsilon^2 \right)\\&
\leq \frac{\al |w|^2}{2}+\frac{3}{2}\leq \frac{\al |w|^2}{2}+\frac{\Re \phi(z)}{2}\leq \frac{\al \Re z}{2}+\frac{\Re \phi(z)}{2}\leq \Re \phi(z),
\end{split}
\end{equation*}
proving that $\Phi(z,w)\in \Ha^2$.

(2) Suppose $\Phi(z_0,w_0)=\Phi(z_1,w_1)$. If $z_0=z_1$ then clearly $w_0=w_1$, so assume $z_0\neq z_1$. Then $a:=\phi(z_0)=\phi(z_1)$ which, up to relabeling, implies $z_0\in U$ and $z_1\in V$. That is, $\|w_0\|^2\leq \Re z_0\leq |z_0|<R_0$ and $|z_1|>R_\infty$. Thus
\begin{equation*}
\frac{\sqrt{\al}\delta w_0}{1+a}+\frac{\epsilon z_0}{1+z_0}=\frac{\sqrt{\al}\delta w_1}{1+a}+\frac{\epsilon z_1}{1+z_1}.
\end{equation*}
Note that for all $(z,w)\in \Ha^2$,
\[
   \left|\frac{\sqrt{\al}\delta w}{1+\phi(z)}\right|^2=\frac{\al \delta^2 |w|^2}{|1+\phi(z)|^2}\leq \frac{\delta^2 \al\Re z}{|1+\phi(z)|^2}\leq \frac{\delta^2\Re \phi(z)}{|1+\phi(z))|^2}\leq \delta^2.
\]
Hence,
\begin{equation*}
\frac{\epsilon R_\infty}{1+R_\infty} - \delta \leq  \left|\frac{\epsilon z_1}{1+z_1}\right|-\left|\frac{\sqrt{\al}\delta w_1}{1+a}\right|\leq \left|\frac{\sqrt{\al}\delta w_1}{1+a}+\frac{\epsilon z_1}{1+z_1}\right|=\left|\frac{\sqrt{\al}\delta w_0}{1+a}+\frac{\epsilon z_0}{1+z_0}\right|\leq
\delta+\frac{\epsilon R_0}{1-R_0}.
\end{equation*}
Therefore,
\[
\epsilon\left(\frac{ R_\infty}{1+R_\infty}-\frac{ R_0}{1-R_0}\right) \leq 2\delta.
\]
But $\frac{ R_\infty}{1+R_\infty}-\frac{ R_0}{1-R_0}=1-r_1/2-r_{-1}/(2-2r_{-1})>1-1/2-1/2>0$, hence, if $\delta$ is sufficiently small the previous condition is never satisfied and the map is univalent.
\end{example}

\end{document}